\documentclass[11pt, notitlepage]{article}
\usepackage{amssymb,amsmath,comment}
\catcode`\@=11 \@addtoreset{equation}{section}
\def\thesection{\arabic{section}}

\def\theequation{\thesection.\arabic{equation}}
\catcode`\@=12
\usepackage{colortbl}%
\usepackage{a4wide}

\newcommand{\fa} {\forall}
\newcommand{\ds} {\displaystyle}
\newcommand{\e}{\epsilon}

\newcommand{\al} {\alpha}
\newcommand{\ba} {\beta}
\newcommand{\de} {\delta}
\newcommand{\ga} {\gamma}
\newcommand{\Ga} {\Gamma}
\newcommand{\Om} {\Omega}
\newcommand{\ra} {\rightarrow}

\newcommand{\De} {\Delta}
\newcommand{\la} {\lambda}

\newcommand{\noi} {\noindent}

\newcommand{\mb} {\mathbb}
\newcommand{\mc} {\mathcal}
\newcommand{\lra} {\longrightarrow}
\newcommand{\ld} {\langle}
\newcommand{\rd} {\rangle}

\usepackage[all]{xy}
\catcode`\@=11
\def\theequation{\@arabic{\c@section}.\@arabic{\c@equation}}
\catcode`\@=12

\def\QED{\hfill {$\square$}\goodbreak \medskip}

\newtheorem{Theorem}{Theorem}[section]
\newtheorem{Lemma}[Theorem]{Lemma}
\newtheorem{Proposition}[Theorem]{Proposition}
\newtheorem{Corollary}[Theorem]{Corollary}
\newtheorem{Remark}[Theorem]{Remark}
\newtheorem{Definition}[Theorem]{Definition}

\begin{document}
{\vspace{0.01in}

\title
{On The Fu\v{c}ik Spectrum Of Non-Local Elliptic Operators}

\author{
{\bf  Sarika Goyal\footnote{email: sarika1.iitd@gmail.com}} and {\bf  K. Sreenadh\footnote{e-mail: sreenadh@gmail.com}}\\
{\small Department of Mathematics}, \\{\small Indian Institute of Technology Delhi}\\
{\small Hauz Khaz}, {\small New Delhi-16, India}\\
 }

\date{}

\maketitle

\begin{abstract}

In this article, we study the Fu\v{c}ik spectrum of fractional
Laplace operator which is defined as the set of all $(\al,\ba)\in
\mb
 R^2$ such that
 \begin{equation*}
 \quad \left.
\begin{array}{lr}
 \quad (-\De)^s u = \al u^{+} - \ba u^{-} \; \text{in}\;
\Om \\
 \quad \quad \quad \quad u = 0 \; \mbox{in}\; \mb R^n \setminus\Om.\\
\end{array}
\quad \right\}
\end{equation*}
has a non-trivial solution $u$, where $\Om$ is a bounded domain in
$\mb R^n$ with Lipschitz boundary, $n>2s$, $s\in(0,1)$. The
existence of a first nontrivial curve $\mc C$ of this spectrum, some
properties of this curve $\mc C$, e.g. Lipschitz continuous,
strictly decreasing and asymptotic behavior are studied in this
article. A variational characterization of second eigenvalue of the
fractional eigenvalue problem is also obtained. At the end, we study
a nonresonance problem with respect to Fu\v{c}ik spectrum.

\medskip

\noi \textbf{Key words:} Non-local operator, fractional Laplacian,
Fu\v{c}ik spectrum, Nonresonance.

\medskip

\noi \textit{2010 Mathematics Subject Classification:} 35R11, 35R09,
35A15.

\end{abstract}

\bigskip
\vfill\eject

\section{Introduction}
The Fu\v{c}ik spectrum of fractional Laplace operator is defined as
the set of all $(\al,\ba)\in \mb
 R^2$ such that
$$
\left\{
\begin{array}{lr}
\quad (-\De)^{s} u = \al u^{+} - \ba u^{-}\;\text{in}\; \Om\\
\quad\quad u = 0 \;\text{on}\; \mb R^n \setminus \Om.
\end{array}
\right.
$$
has a non-trivial solution $u$, where $s\in (0,1)$ and $(-\De)^{s}$
be fractional Laplacian operator defined as
\begin{equation*}
(-\De)^{s} u(x)= -\frac{1}{2} \int_{\mb
R^n}\frac{u(x+y)+u(x-y)-2u(x)}{|y|^{n+2s}} dy \;\text{for all} \;
x\in \mb R^n.
\end{equation*}

\noi In general, we study the Fu\v{c}ik spectrum of an equation
driven by the non-local operator $\mc L_{K}$ which is defined as
\[\mc L_K u(x)= \frac12 \int_{\mb R^n}(u(x + y) + u(x- y) -2u(x))K(y)
dy\;\;\text{for\; all}\;\; x\in \mb R^n,\]
where $K :\mb R^n\setminus\{0\}\ra(0,\infty)$ satisfies the following:\\
$(i)\; mK \in L^1(\mb R^n),\;\text{where}\; m(x) = \min\{|x|^2,
1\}$,\\
$(ii)$ There exist $\la>0$ and $s\in(0,1)$ such that $K(x)\geq \la
|x|^{-(n+2s)},$\\
$(iii)\; K(x) = K(-x)$ for any $ x\in \mb R^n\setminus\{0\}$.

\noi In case $K(x) = |x|^{-(n+2s)}$, $\mc L_K$ is the fractional
Laplace operator $-(-\De)^s$.\\

\noi The Fu\v{c}ik spectrum of the non-local operator $\mc L_{K}$ is
defined as the set $\sum_{K}$ of $(\al,\ba)\in \mb
 R^2$ such that
\begin{equation}\label{eq01}
 \quad \left.
\begin{array}{lr}
 \quad -\mc L_{K}u = \al u^{+} - \ba u^{-} \; \text{in}\;
\Om \\
 \quad \quad \quad \quad u = 0 \; \mbox{in}\; \mb R^n \setminus\Om.\\
\end{array}
\quad \right\}
\end{equation}

\noi \eqref{eq01} has a nontrivial solution $u$. Here $u^{\pm} =
\max(\pm u, 0)$ and $\Om\subset \mb R^n$ is a bounded
 domain with Lipshitz boundary. For $\al=\ba$, Fu\v{c}ik spectrum of \eqref{eq01} becomes the usual spectrum of
\begin{equation}\label{eq02}
 \quad \left.
\begin{array}{lr}
 \quad -\mc L_{K}u = \la u \; \text{in}\;
\Om \\
 \quad \quad u = 0 \; \mbox{in}\; \mb R^n \setminus\Om.\\
\end{array}
\quad \right\}
\end{equation}
 Let $0<\la_1<\la_2\leq...\leq\la_k\leq...$ denote the distinct eigenvalues of
  \eqref{eq02}. It is proved in \cite{var} that the first eigenvalue $\la_1$ of \eqref{eq02} is simple,
isolated and can be characterized as follows
\[\la_1 = \inf_{u\in X_{0}}\left\{\int_{Q}(u(x)-u(y))^2 K(x-y)dxdy : \int_{\Om} u^2=1\right\}.\]
The author also proved that the eigenfunction corresponding to
$\la_1$ are of constant sign. We observe that $\sum_{K}$ clearly
contains $(\la_k,\la_k)$ for each $k\in \mb N$ and two lines
$\la_1\times\mb R$ and $\mb R\times\la_1$. $\sum_{K}$ is symmetric
with respect to diagonal. In this paper we will prove that the two
lines $\mb R\times \la_1$ and $\la_1\times \mb R$ are isolated in
$\sum_{K}$ and give a variational characterization of second
eigenvalue $\la_2$ of $-\mc L_K$.

\noi When $s=1$, the fractional Laplacian operator become the usual
Laplace operator. Fu\v{c}ik spectrum is introduced by Fu\v{c}ik in
1976. The negative Laplacian in one dimension with periodic boundary
condition is studied in \cite{fu}. Also study of Fu\v{c}ik spectrum
in case of Laplacian, p-Laplacian equation with Dirichlet, Neumann
and robin boundary condition has been studied by many authors \cite{
al, ar, cfg, cg, fg, ro, ros, ap, pe, se}. A nonresonance problem
with respect to Fu\v{c}ik spectrum is also discussed in many papers
\cite{cfg, pr, mp}. To best of our knowledge, no work has been done
to find the Fu\v{c}ik spectrum for non-local operator. Recently a
lot of attention is given to the study of fractional and non-local
operator of elliptic type due to concrete real world applications in
finance, thin obstacle problem, optimization, quasi-geostrophic flow
etc \cite{mp, ls, weak, low}. Here we use the similar approach to
find Fu\v{c}ik spectrum that is used in \cite{cfg}.

\noi In \cite{mp}, Servadei and Valdinoci discussed Dirichlet
boundary value problem in case of fractional Laplacian using the
Variational techniques. We also used the similar variational
technique to find $\sum_{K}$. Due to non-localness of the fractional
Laplacian, the space $(X_0,\|.\|_{X_0})$ is introduced by Servadei.
We introduce this space as follows:
\[X= \left\{u|\;u:\mb R^n \ra\mb R \;\text{is measurable},
u|_{\Om} \in L^2(\Om)\;
 and\;  \left(u(x)- u(y)\right)\sqrt{K(x-y)}\in
L^2(Q)\right\},\]

\noi where $Q=\mb R^{2n}\setminus(\mc C\Om\times \mc C\Om)$ and
 $\mc C\Om := \mb R^n\setminus\Om$. The space X is endowed with the norm defined as
\begin{align*}\label{eq44}
 \|u\|_X = \|u\|_{L^2(\Om)} +\left( \int_{Q}|u(x)-u(y)|^{2}K(x-y)dx
dy\right)^{\frac12}.
\end{align*}
 Then we define
 \[ X_0 = \{u\in X : u = 0 \;\text{a.e. in}\; \mb R^n\setminus \Om\}\]
with the norm
\[\|u\|_{X _0}=\left( \int_{Q}|u(x)-u(y)|^{2}K(x-y)dx
dy\right)^{\frac12}\] is an Hilbert space. Note that the norm
$\|.\|_{X_0}$ involves the interaction between $\Om$ and $\mb
R^n\setminus\Om$.
\begin{Remark}
 $(i)$ $C_{c}^{2}(\Om)\subseteq X_{0}$, $X\subseteq H^s(\Om)$ and
$X_0\subseteq H^s(\mb R^n)$, where $H^s(\Om)$ denotes the usual
fractional Sobolev space endowed with the norm
\[\|u\|_{H^{s}(\Om)}=\|u\|_{L^2}+ \left(\int_{\Om\times\Om} \frac{(u(x)-u(y))^{2}}{|x-y|^{n+2s}}dxdy \right)^{\frac 12}.\]

\noi $(ii)$ The embedding $X_{0}\hookrightarrow L^{2^*}(\mb
R^n)=L^{2^*}(\Om)$ is continuous, where $2^*=\frac{2n}{n-2s}$. To
see the detailed of these embeddings, one can refer \cite{hi, mp}.
\end{Remark}
\begin{Definition}\label{def1}
 A function $u \in X_0$ is a weak solution of \eqref{eq01}, if for every $v\in X_0$, $u$
satisfies
\begin{align*}
\int_{Q}(u(x)-u(y))(v(x)-v(y)) K(x-y)dxdy = \al \int_{\Om} u^{+} v
dx - \ba \int_{\Om}u^{-} v dx,
\end{align*}
\end{Definition}
\noi Weak solutions of \eqref{eq01} are exactly the critical points
of the functional $J : X_0 \ra \mb R$ defined as
\[  J (u) = \frac12
\int_{Q}(u(x)-u(y))^{2} K(x-y) dx dy - \frac{\al}{2} \int_{\Om}
(u^{+})^2dx - \frac{\ba}{2} \int_{\Om} (u^{-})^2dx.\] $J$ is
Fr$\acute{e}$chet differentiable in $ X_0$ and
\[\langle J^\prime (u),\phi\rangle=\int_{Q}(u(x)-u(y))(\phi(x)-\phi(y)) K(x-y) dx dy - \al \int_{\Om}
u^{+} \phi dx - \ba \int_{\Om} u^{-} \phi dx.\]

\noi The paper is organized as follows: In section 2 we construct a
first nontrivial curve in $\sum_{K}$, described as $(p+c(p),c(p))$.
In section 3 we prove that the lines $\la_1\times \mb R$ and $\mb
R\times\la_1$ are isolated in $\sum_{K}$, the curve that we obtained
in section 2 is the first nontrivial curve and give the variational
characterization of second eigenvalue of $-\mc L_K$. In section 4 we
prove some properties of the first curve. A non-resonance problem
with respect to Fu\v{c}ik spectrum is also studied in section 5.

\noi We shall throughout use the function space $X_0$ with the norm
$\|.\|_{X_0}$ and we use the standard $L^{p}(\Om)$ space whose norms
are denoted by $\|u\|_{L^p}$. Also $\phi_1$ is the eigenfunction of
$-\mc L_{K}$ corresponding to $\la_1$.

\section{Fu\v{c}ik Spectrum $\sum_{K}$ for $-\mc L_K$}
In this section we study the existence of first nontrivial curve in
the Fu\v{c}ik spectrum $\sum_{K}$ of $-\mc L_K$. We find that the
points in $\sum_{K}$ are associated with the critical value of some
restricted functional.

\noi For  this we fix $p\in \mb R$ and for $p\geq 0$, consider the
functional $J_{p}: X_0\ra \mb R $ by
\begin{align}
J_{p}(u)= \int_{Q}(u(x)-u(y))^2 K(x-y)dxdy - p\int_{\Om} (u^{+})^2
dx.
\end{align}
Then $J_{p}\in C^{1}(X_0,\mb R)$ and for any $\phi\in X_0$
\[\langle J_{p}^{\prime}(u),\phi \rangle = 2 \int_{Q}
(u(x)-u(y))(\phi(x)-\phi(y))K(x-y)dx dy - 2p \int_{\Om}
u^{+}(x)\phi(x)dx.\]

\noi Also $\tilde{J_{p}}:= J_{p}|_{\mc P}$ is $C^1(X_0,\mb R)$,
where $\mc P$ is defined as \[\mc P=\left\{u\in X_0 : I(u):=
\int_{\Om} u^2 dx =1\right\}.\]

\noi We first note that $u\in\mc P$ is a critical point of
$\tilde{J_{p}}$ if and only if there exists $t\in \mb R$ such that
\begin{align}\label{eq2}
\int_{Q}(u(x)-u(y))(v(x)-v(y)) K(x-y)dxdy - p\int_{\Om} u^{+} v dx =
t\int_{\Om}u v dx,
\end{align}
for all $v\in X_0$. Hence $u\in \mc P$ is a nontrivial weak solution
of the problem
\begin{align*}
-\mc L_{K}u &= (p+t) u^{+} - t u^{-} \; \text{in}\;
\Om ;\\
 u &= 0 \; \mbox{on}\; \mb R^n \setminus\Om,
 \end{align*}
which exactly means $(p+t,t)\in \sum_{K}$. Putting $v=u$ in
\eqref{eq2}, we get $t= \tilde{J_{p}}(u)$. Thus we have the
following result, which describe the relationship between the
critical points of $\tilde{J_{p}}$ and the spectrum $\sum_{K}$.
\begin{Lemma}
For $p\geq 0$, $(p+t,t)\in \mb R^2$ belongs to the spectrum
$\sum_{K}$ if and only if there exists a critical point $u\in \mc P$
of $\tilde{J_{p}}$ such that $t= \tilde{J_{p}}(u)$, a critical
value.
\end{Lemma}


\noi Now we look for the minimizers of $\tilde{J_{p}}$.

\begin{Proposition}
The first eigenfunction $\phi_1$ is a global minimum for
$\tilde{J_{p}}$ with $\tilde{J_{p}}(\phi_1)=\la_1-p$. The
corresponding point in $\sum_{K}$ is $(\la_1,\la_1-p)$ which lies on
the vertical line through $(\la_1,\la_1)$.
\end{Proposition}

\begin{proof}
It is easy to see that $\tilde{J_{p}}(\phi_1) = \la_1 - p$ and
\begin{align*}
\tilde{J_{p}}(u)=& \int_{Q}(u(x)-u(y))^2 K(x-y)dxdy -
p\int_{\Om} (u^{+})^2 dx\\
 \geq &  \la_1 \int_{\Om} u^2 dx - p\int_{\Om} (u^{+})^2 dx
 \geq \la_1 -p.
\end{align*}
Thus $\phi_1$ is a global minimum of $\tilde{J_{p}}$ with
$\tilde{J_{p}}(\phi_1)=\la_1-p$.\QED
\end{proof}
\noi Now we have a second critical point of $\tilde{J_{p}}$ at
$-\phi_1$ corresponding to strict local minimum.
\begin{Proposition}
The negative eigenfunction $-\phi_1$ is a strict local minimum for
$\tilde{J_{p}}$ with $\tilde{J_{p}}(-\phi_1)=\la_1$. The
corresponding point in $\sum_{K}$ is $(\la_1+p,\la_1)$, which lies
on the horizontal line through $(\la_1,\la_1)$.
\end{Proposition}

\begin{proof}
Let us suppose by contradiction that there exists a sequence $u_k\in
\mc P$, $u_k\ne -\phi_1$ with $\tilde{J_{p}}(u_k) \leq \la_1$,
$u_k\ra -\phi_1$ in $X_0$. Firstly, we show that $u_k$ changes sign
for sufficiently large $k$. Since $u_k\ne -\phi_1$, it must be $\leq
0$ for some $x\in X_0$. If $u_k \leq 0$ for a.e $x\in \Om$, then
\begin{align}
\tilde{J_{p}}(u_k)= \int_{Q}(u_k(x)-u_k(y))^2 K(x-y) dxdy
>\la_1,
\end{align}
since $u_k\ne \pm\phi_1$ and we get contradiction as
$\tilde{J_{p}}(u_k) \leq \la_1$. So $u_k$ changes sign for
sufficiently large $k$. Define $w_k :=
\frac{u_{k}^{+}}{\|u_{k}^{+}\|_{L^{2}}}$ and
\[r_k := \int_{Q}(w_{k}(x)-w_{k}(y))^2 K(x-y)dxdy. \]

\noi Now we claim that $r_k \ra \infty$. Let us suppose by
contradiction that $r_k$ is bounded. Then there exists a subsequence
of $w_k$ still denoted by $w_k$ and $w\in X_0$ such that $w_k
\rightharpoonup w$ weakly in $X_0$ and $w_k\ra w$ strongly in
$L^{2}(\mb R^n)$. Therefore $\int_{\Om} w^2 dx =1$, $w\geq 0$ a.e.
and so for some $\e>0$, $\de=|\{x\in X_0: w(x)\geq \e\}|>0$. As
$u_k\ra -\phi_1$ in $X_0$ and hence in $L^{2}(\Om)$. Therefore
$|\{x\in \Om : u_k(x)\geq \e\}|\ra 0$ as $k\ra \infty$ and so
$|\{x\in \Om : w_k(x)\geq \e\}|\ra 0$ as $k\ra\infty$ which is a
contradiction as $\de>0$ . Hence the claim. Next,
\begin{align*}
(u_k(x)-&u_k(y))^2=
((u_{k}^{+}(x)-u_{k}^{+}(y))-(u_{k}^{-}(x)-u_{k}^{-}(y)))^2\notag\\
=&(u_{k}^{+}(x)-u_{k}^{+}(y))^2+ (u_{k}^{-}(x)-u_{k}^{-}(y))^2 - 2
(u_{k}^{+}(x)-u_{k}^{+}(y))(u_{k}^{-}(x)-u_{k}^{-}(y))\notag\\
=& (u_{k}^{+}(x)-u_{k}^{+}(y))^2+ (u_{k}^{-}(x)-u_{k}^{-}(y))^2 +
 2u_{k}^{+}(x) u_{k}^{-}(y)+ 2 u_{k}^{-}(x) u_{k}^{+}(y),
 \end{align*}
 where we have used $u_{k}^{+}(x) u_{k}^{-}(x)=0$. Using $K(x)=K(-x)$ we have
 \begin{align}\label{e12}
\int_{Q}u_{k}^{+}(x) u_{k}^{-}(y) K(x-y)dxdy= \int_{Q} u_{k}^{+}(y)
u_{k}^{-}(x) K(x-y)dxdy.
\end{align}
 Then from above estimates, we get
\begin{align*}
 \tilde{J_{p}}(u_k)=& \int_{Q}(u_k(x)-u_k(y))^2 K(x-y)dxdy -
p\int_{\Om} (u^{+}_k)^2 dx\\
=& \int_{Q}(u_{k}^{+}(x)-u_{k}^{+}(y))^2 K(x-y)dxdy +
\int_{Q}(u_{k}^{-}(x)-u_{k}^{-}(y))^2 K(x-y)dxdy\\
& + 4\int_{Q}u_{k}^{+}(x) u_{k}^{-}(y) K(x-y)dxdy - p\int_{\Om}
(u^{+}_{k})^2 dx\\
\geq & (r_k-p) \int_{\Om} (u_{k}^{+})^2 dx + \la_1
\int_{\Om}(u_{k}^{-})^2 dx + 4\int_{Q}u_{k}^{+}(x)
u_{k}^{-}(y) K(x-y)dxdy \\
\geq & (r_k-p) \int_{\Om} (u_{k}^{+})^2 dx + \la_1
\int_{\Om}(u_{k}^{-})^2 dx.
\end{align*}
As $u_k\in \mc P$, we get
\[\tilde{J_{p}}(u_k)\leq \la_1 = \la_1 \int_{\Om} (u_{k}^{+})^2 dx+
\la_1 \int_{\Om} (u_{k}^{-})^2 dx.\] \noi Combining both the
inequalities we have,
\[(r_k -p) \int_{\Om} (u_{k}^{+})^2 dx + \la_1 \int_{\Om}
(u_{k}^{-})^2 dx \leq \la_1 \int_{\Om} (u_{k}^{+})^2 dx + \la_1
\int_{\Om} (u_{k}^{-})^2 dx.\] \noi This implies $(r_k -p-\la_1)
\int_{\Om} (u_{k}^{+})^2 dx\leq 0,$ and hence $r_k- p\leq \la_1$,
\noi which contradicts the fact that $r_k\ra +\infty$, as
required.\QED
\end{proof}

We will now find the third critical point based on mountain pass
Theorem. A norm of derivative of the restriction $\tilde{J_{p}}$ of
$J_{p}$ at $u\in \mc P$ is defined as
\[\|\tilde{J_{p}}(u)\|_{*}=\min\{\|\tilde{J}_{p}^{\prime}(u)- t I^{\prime}(u)\|_{X_0}:  t\in \mb R\}.\]
\begin{Definition}
We say that $J_{p}$ satisfies the Palais-Smale (in short, (P.S))
condition on $\mc P$ if for any sequence $u_k\in \mc P$ such that
$J_{p}(u_k)$ is bounded and $\|\tilde{J}^{\prime}_{p}(u_k)\|_{*} \ra
0$, then there exists a subsequence that converges strongly in
$X_0$.
\end{Definition}
Now we state here the version of mountain pass theorem, that will be
used later.
\begin{Proposition}\label{pr01}
Let $E$ be a Banach space, $g,f \in C^{1}(E,\mb R)$, $M=\{u\in E
\;|\; g(u)=1\}$ and $u_0$, $u_1\in M$. let $\e>0$ such that
$\|u_1-u_0\|>\e$ and \[\inf\{f(u): u\in M \;\mbox{and}\;
\|u-u_0\|_{E}=\e\}>\max\{f(u_0),f(u_1)\}.\]

\noi Assume that $f$ satisfies the $(P.S)$ condition on $M$ and that
\[\Gamma =\{\ga \in C([-1,1], M): \ga(-1)=u_0 \;\mbox{and}\; \ga(1)=u_1\}\]
is non empty. Then $c=\inf_{\ga \in \Gamma}\max_{u\in\ga[-1,1]}
f(u)$ is a critical value of $f|_M$.
\end{Proposition}

\begin{Lemma}\label{le21}
$J_{p}$ satisfies the $(P.S)$ condition on $\mc P$.
\end{Lemma}

\begin{proof}
Let $\{u_k\}$ be a $(P.S)$ sequence. i.e., there exists $K>0$ and
$t_k$ such that
\begin{align}
|J_{p}(u_k)|&\leq K,\label{eq3}\\
\int_{Q}(u_k(x)-u_k(y))(v(x)-v(y))& K(x-y)dxdy -  p\int_{\Om}
u^{+}_{k} v - t_k \int_{\Om} u_{k}v  =o_k(1)\|v\|_{X_0}.\label{eq4}
\end{align}
\noi From \eqref{eq3}, we get $u_k$ is bounded in $X_0$. So we may
assume that up to a subsequence $u_k\rightharpoonup u_0$ weakly in
$X_0$, and $u_{k}\ra u_0$ strongly in $L^{2}(\Om)$. Putting $v=u_k$
in \eqref{eq4}, we get $t_k$ is bounded and up to a subsequence
$t_k$ converges to $t$.
 We now claim that $u_k\ra u_0$ strongly in $X_0$. As
$u_k\rightharpoonup u_0$ weakly in $X_0$, we have
\begin{align}\label{eq8}
\int_{Q}(u_k(x)-u_k(y))&(v(x)-v(y)) K(x-y)dxdy\notag\\
&\ra \int_{Q} (u_0(x)-u_0(y))(v(x)-v(y)) K(x-y)dxdy
\end{align}
for all $v\in X_0$. Also $\tilde{J_{p}^{\prime}}(u_k)(u_k-u_0)=
o_k(1)$. Therefore we get
\begin{align*}
\left|\int_{Q}(u_k(x)-u_k(y))^2\right.& \left.K(x-y)dxdy - \int_{Q}
(u_k(x)-u_k(y))(u_0(x)-u_0(y)) K(x-y)dxdy\right|\notag\\
&\leq O(\e_k)+p\|u_{k}^{+}\|_{L^2}
\|u_k-u_0\|_{L^2}+|t_k|\|u_k\|_{L^2}
\|u_k-u_0\|_{L^2}\notag\\
& \ra 0\;\mbox{as}\; k\ra\infty.
\end{align*}

\noi Taking $v=u_0$ in \eqref{eq8}, we get
\begin{align*}
\int_{Q}(u_k(x)-u_k(y))(u_0(x)-u_0(y)) K(x-y)dxdy\ra \int_{Q}
(u_0(x)-u_0(y))^2 K(x-y)dxdy.
\end{align*}
From above two equations, we have
\begin{align*}
\int_{Q}(u_k(x)-u_k(y))^2 K(x-y)dxdy\ra \int_{Q} (u_0(x)-u_0(y))^2
K(x-y)dxdy.
\end{align*}
Thus $\|u_k\|^{2}_{X_0} \ra \|u_0\|^{2}_{X_0}$. Now using this and
$v=u_0$ in \eqref{eq8}, we get
\begin{align*}
\|u_k-u_0\|^{2}_{X_0} &= \|u_k\|^{2}_{X_0}+\|u_k\|^{2}_{X_0}-2\int_{Q}(u_k(x)-u_k(y))(u_0(x)-u_0(y)) K(x-y)dxdy\\
&\lra 0\; \mbox{as}\; k\ra \infty.
\end{align*}
Hence $u_k\ra u_0$ strongly in $X_0$.\QED
\end{proof}

\begin{Lemma}\label{le22}
Let $\e_0>0$  be such that
\begin{align*}
\tilde{J_{p}}(u)>\tilde{J_{p}}(-\phi_1)
\end{align*}
for all $u\in
B(-\phi_1, \e_0)\cap \mc P$ with $u\ne -\phi_1$, where the ball is
taken in $X_0$. Then for any $0<\e<\e_0$,
\begin{align}\label{eq14}
\inf\{\tilde{J_{p}}(u): u\in \mc P\;\mbox{and}\;
\|u-(-\phi_1)\|_{X_0}=\e\}> \tilde{J_{p}}(-\phi_1).
\end{align}
\end{Lemma}

\begin{proof}
Assume by contradiction that infimum in \eqref{eq14} is equal to
$\tilde{J_{p}}(-\phi_1)=\la_1$ for some $\e$ with $0<\e<\e_0$. Then
there exists a sequence $u_k\in \mc P$ with
$\|u_k-(-\phi_1)\|_{X_0}=\e$ such that
\begin{align*}
\tilde{J_{p}}(u_k)\leq \la_1 + \frac{1}{2k^2}.
\end{align*}
Consider the set $C =\{u\in\mc P: \e-\de
\leq\|u-(-\phi_1)\|_{X_0}\leq \e+\de\}$, where $\de$ is chosen such
that $\e-\de>0$ and $\e+\de<\e_0$. From \eqref{eq14} and given
hypotheses, it follows that $\inf\{\tilde{J_{p}}(u): u\in
C\}=\la_1.$ Now for each $k$, we apply Ekeland's variational
principle to the functional $\tilde{J_{p}}$ on $C$ to get the
existence of $v_k\in C$ such that
\begin{align*}
\tilde{J_{p}}(v_k)&\leq \tilde{J_{p}}(u_k), \quad
\|v_k-u_k\|_{X_0} \leq \frac{1}{k},\\
\tilde{J_{p}}(v_k)&\leq \tilde{J_{p}}(u) +
\frac{1}{k}\|u-v_k\|_{X_0}\;\fa\; u\in C \label{eq020}
\end{align*}

\noi We claim that $v_k$ is a (P.S) sequence for $\tilde{J_{p}}$ on
$\mc P$ i.e. $\tilde{J_{p}}(v_k)$ is bounded and
$\|\tilde{J}^{\prime}_{p}(v_k)\|_{*}\ra 0$. Once this is proved we
get by Lemma \ref{le21}, up to a subsequence $v_k\ra v$ in $X_0$.
Clearly $v\in \mc P$ and satisfies $\|v-(-\phi_1)\|_{X_0}\leq
\e+\de<\e_0$ and $\tilde{J_{p}}(v)= \la_1$ which contradicts the
given hypotheses. Then proof of the claim can be proved similar as
in Lemma 2.9 of \cite{cfg} by replacing $\|.\|_{1,p}$ by
$\|.\|_{X_0}$.\QED
\end{proof}

\begin{Proposition}\label{prop2.8}
 Let $\e>0$ such that
$\|\phi_1-(-\phi_1)\|_{X_0}>\e$ and
\[\inf\{\tilde{J_p}(u): u\in \mc P \;\mbox{and}\;
\|u-(-\phi_1)\|_{X_0}=\e\}>\max\{\tilde{J_p}(-\phi_1),
\tilde{J_p}(\phi_1)\}.\] Then
$\Gamma =\{\ga \in C([-1,1], \mc P): \ga(-1)= -\phi_1 \;\mbox{and}\;
\ga(1)=\phi_1\}$
is non empty and
\begin{align}\label{eq18}
c(p)=\inf_{\ga \in \Gamma}\max_{u\in\ga[-1,1]} J_{p}(u)
\end{align}
is a critical value of $\tilde{J_{p}}$. Moreover $c(p)>\la_1$.
\end{Proposition}

\begin{proof} Let $\phi\in X_0$ be
such that $\phi\not\in \mb R\phi_1$ and consider the path
$\ga(t)=\frac{t\phi_1+(1-|t|)\phi}{\|t\phi_1+(1-|t|)\phi\|_{L^2}}$,
then $\ga(t)\in \Ga$ . Moreover by Lemmas \ref{le21} and \ref{le22},
$\tilde{J_p}$ satisfies (P.S) condition and the geometric
assumptions. Then by Proposition \ref{pr01}, $c(p)$ is a critical
value of $\tilde{J_p}$. Using the definition of $c(p)$ we have
$c(p)>\max\{\tilde{J_p}(-\phi_1), \tilde{J_p}(\phi_1)\}=\la_1$.\QED
\end{proof}
\noi Thus we have proved the following:
\begin{Theorem}\label{th1}
For each $p\geq 0$, the point $(p+c(p),c(p))$, where $c(p)>\la_1$ is
defined by the minimax formula \eqref{eq18}, then the point
$(p+c(p),c(p))$ belongs to $\sum_{K}$.
\end{Theorem}

\noi It is a trivial fact that $\sum_{K}$ is symmetric with respect
to diagonal. The whole curve, that we obtain using Theorem \ref{th1}
and symmetrizing, is denoted by
\[\mc C:= \{(p+c(p),c(p)),(c(p),p+c(p)): p\geq 0\}.\]

\section{First Nontrivial Curve}
We start this section by establishing that the lines $\mb R\times
\{\la_1\}$ and $\{\la_1\}\times \mb R$ are isolated in $\sum_{K}$.
Then we state some topological properties of the functional
$\tilde{J_{p}}$ and finally we prove that the curve $\mc C$
constructed in the previous section is the first nontrivial curve in
the spectrum $\sum_{K}$.

\begin{Proposition}\label{pr1}
The lines $\mb R\times \{\la_1\}$ and $\{\la_1\}\times \mb R$ are
isolated in $\sum_{K}$. In other words, there exists no sequence
$(\al_k, \ba_k)\in \sum_{K}$ with $\al_k
> \la_1$ and $\ba_k > \la_1$ such that $(\al_k, \ba_k)\ra (\al,
\ba)$ with $\al=\la_1$ or $\ba=\la_1$.
\end{Proposition}

\begin{proof}
Suppose by contradiction that there exists a sequence $(\al_k,
\ba_k)\in \sum_{K}$ with $\al_k,$ $\ba_k>\la_1$ and $(\al_k,
\ba_k)\ra (\al, \ba)$ with $\al$ or $\ba=\la_1$. Let $u_k\in X_0$ be
a solution of
\begin{align}\label{eq17}
-\mc L_{K}u_k = \al_k u^{+}_k - \ba_k u^{-}_k \;\mbox{in}\;\Om,
\quad u_k=0 \;\mbox{on}\; \mb R^n\setminus \Om
\end{align}
 with $\|u_k\|_{L^2}=1$. Then we have
 \[\int_{Q}(u_{k}(x)-u_{k}(y))^2 K(x-y)dxdy = \al_k\int_{\Om}(u_k^+)^2 dx - \ba_k\int_{\Om}(u_k^-)^2 dx\leq \al_k,\]
which shows that $u_k$ is bounded sequence in $X_0$. Therefore up to
a subsequence $u_k \rightharpoonup u$ weakly in $X_0$ and $u_k\ra u$
strongly in $L^2(\Om)$. Then the limit $u$ satisfies
\[\int_{Q}(u(x)-u(y))^2 K(x-y)dxdy = \la_1\int_{\Om}(u^+)^2 dx - \ba\int_{\Om}(u^-)^2 dx,\]
since $u_k \rightharpoonup u$ weakly in $X_0$ and
$\ld\tilde{J}^{\prime}_{p}(u_k),u_k-u\rd\ra 0$ as $k\ra \infty$. i.e
$u$ is a weak solution of
\begin{align}\label{eq017}
-\mc L_{K}u = \al u^{+} - \ba u^{-}\;\mbox{in}\;\Om,\;\quad \; u=0
\;\mbox{on}\;\mb R^n\setminus \Om
\end{align}
where we have consider the case $\al=\la_1$. Multiplying by $u^+$ in
\eqref{eq017}, integrate, using
\[(u(x)-u(y))(u^{+}(x)-u^{+}(y))=(u^{+}(x)-u^{+}(y))^2+ u^{+}(x)u^{-}(y)+ u^{+}(y)u^{-}(x)\]
and \eqref{e12}, we get
\[\int_{Q}(u^{+}(x)-u^{+}(y))^2 K(x-y)dxdy + 2\int_{Q}u^{+}(x) u^{-}(y) K(x-y)dxdy = \la_1\int_{\Om}(u^+)^2 dx. \]
Using this we have,
\begin{align*}
\la_1\int_{\Om}(u^+)^2 dx\leq \int_{Q}(u^{+}(x)-u^{+}(y))^2
K(x-y)dxdy \leq \la_1\int_{\Om}(u^+)^2 dx
\end{align*}
Thus \[\int_{Q}(u^{+}(x)-u^{+}(y))^2 K(x-y)dxdy=
\la_1\int_{\Om}(u^+)^2 dx,\] so that either $u^+\equiv 0$ or
$u=\phi_1$. If $u^+\equiv 0$ then $u\leq 0$ and \eqref{eq017}
implies that $u$ is an eigenfunction with $u\leq 0$ so that $u=
-\phi_1$. So in any case $u_k$ converges to either $\phi_1$ or
$-\phi_1$ in $L^{p}(\Om)$. Thus for every $\e>0$
\begin{align}\label{eq019}
\mbox{either} \;|\{x\in \Om : u_k(x)\leq \e\}|\ra 0\; \mbox{ or}\;
|\{x\in\Om :u_k(x)\geq \e\}|\ra 0.
\end{align}
 On the other hand, taking $u_k^+$ as test function in
\eqref{eq17}, we get
\begin{align*}
\int_{Q}(u^{+}_k(x)-u^{+}_k(y))^2 K(x-y)dxdy+
2\int_{Q}u^{+}_k(x)u^{-}_k(y) K(x-y)dxdy =
\al_k\int_{\Om}(u_{k}^+)^2 dx.
\end{align*}
 Using this, H\"{o}lders inequality and Sobolev
embeddings we get
\begin{align*}
\int_{Q}(u^{+}_k(x)-&u^{+}_k(y))^2 K(x-y)dxdy\\
&\leq \int_{Q}(u^{+}_k(x)-u^{+}_k(y))^2 K(x-y)dxdy+
2\int_{Q}u^{+}_k(x)u^{-}_k(y) K(x-y)dxdy\\
&= \al_k\int_{\Om}(u_{k}^+)^2 dx\\
&\leq \al_k C |\{x\in \Om :
u_k(x)>0\}|^{1-\frac{2}{q}}\|u_k^+\|_{X_0}^{2}
\end{align*}
 with a constant $C>0$, $2<q\leq 2^*=\frac{2n}{n-2s}$. Then we have
\[ |\{x\in \Om :
u_k(x)>0\}|^{1-\frac{2}{q}} \geq \al_k ^{-1}C^{-1}.\] Similarly, one
can show that
\[|\{x\in \Om :
u_k(x)<0\}|^{1-\frac{2}{q}}\geq \ba_k ^{-1}C^{-1}.\] Since
$(\al_k,\ba_k)$ does not belong to the trivial lines $\la_1\times\mb
R$ and $\mb R\times \la_1$ of $\sum_{K}$, by using \eqref{eq17} we
have that $u_k$ changes sign. Hence, from the above inequalities, we
get a contradiction with \eqref{eq019}. Hence the trivial lines
$\la_1\times \mb R$ and $\mb R\times \la_1$ are isolated in
$\sum_{K}$.
\end{proof}
\begin{Lemma}\label{le1}\cite{cfg}
Let $\mc P= \{u\in X_0 : \int_{\Om}u^2 =1\}$ then
\begin{enumerate}
\item $\mc P$ is locally arcwise connected.
\item Any open connected subset $\mc O$ of $\mc P$ is arcwise
connected.
\item If $\mc O^{'}$ is any connected component of an open set $\mc O\subset \mc P$, then $\partial \mc O^{\prime}\cap \mc O=
\emptyset$.
\end{enumerate}
\end{Lemma}

\begin{Lemma}\label{le01}
Let $\mc O=\{u\in \mc P : \tilde{J_p}(u)<r\}$, then any connected
component of $\mc O$ contains a critical point of $\tilde{J_p}$.
\end{Lemma}

\begin{proof}  Proof follows in the same lines as  Lemma 3.6 of \cite{cfg} by replacing $\|.\|_{1,p}$ by $\|.\|_{X_0}.$
\QED
\end{proof}
\begin{Theorem}\label{3.4}
Let $p\geq 0$ then the point $(p+c(p),c(p))$ is the first nontrivial
point in the intersection between $\sum_{K}$ and the line
$(p,0)+t(1,1)$.
\end{Theorem}

\begin{proof}
Assume by contradiction that there exists $\mu$ such that
$\la_1<\mu<c(p)$ and $(p+\mu,\mu)\in \sum_{K}$. Using the fact that
$\{\la_1\}\times \mb R$ and $\mb R\times \{\la_1\}$ are isolated in
$\sum_{K}$ and $\sum_{K}$ is closed we can choose such a point with
$\mu$ minimum. In other words $\tilde{J_{p}}$ has a critical value
$\mu$ with $\la_1<\mu<c(p)$, but there is no critical value in
$(\la_1,\mu)$. If we construct a path connecting from $-\phi_1$ to
$\phi_1$ such that $\tilde{J_p}\leq \mu$, then we get a
contradiction with the definition of $c(p)$, which completes the
proof.

\noi Let $u\in \mc P$ be a critical point of $\tilde{J_p}$ at level
$\mu$. Then $u$ satisfies,
\[ \int_{Q} (u(x)-u(y))(v(x)-v(y)) K(x-y) dxdy = (p+\mu)\int_{\Om}u^{+} vdx -\mu \int_{\Om}u^{-}vdx.\]
for all $v\in X_0$. Replacing $v$ by $u^{+}$ and $u^-$, we have
\[ \int_{Q} (u^+(x)-u^{+}(y))^2 K(x-y) dxdy+ 2 \int_{Q}
u^+(x)u^-(y)K(x-y) dxdy= (p+\mu)\int_{\Om}(u^{+})^2dx, \] and
\[ \int_{Q} (u^-(x)-u^-(y))^2 K(x-y) dxdy+ 2 \int_{Q}
u^+(x)u^-(y)K(x-y) dxdy= \mu\int_{\Om}(u^{-})^2dx .\] Thus we
obtain,
\[\tilde{J_{p}}(u)=\mu,\quad \; \tilde{J_{p}}\left(\frac{u^+}{\|u^+\|_{L^2}}\right)
= \mu -\frac{2 \int_{Q} u^+(x)u^-(y)K(x-y) dxdy}{\|u^+\|^2_{L^2}},\]
\[\tilde{J_{p}}\left(\frac{u^-}{\|u^-\|_{L^2}}\right)= \mu -p -\frac{2 \int_{Q} u^+(x)u^-(y)K(x-y) dxdy}{\|u^-\|^{2}_{L^2}},\]
and
\[\tilde{J_{p}}\left(-\frac{u^-}{\|u^-\|_{L^2}}\right)= \mu -\frac{2 \int_{Q} u^+(x)u^-(y)K(x-y) dxdy}{\|u^-\|^{2}_{L^2}}.\]
Since $u$ changes sign, the following paths are well-defined on $\mc
P$:
\[u_1(t)=\frac{(1-t)u+ t u^+}{\|(1-t)u+ tu^+\|_{L^2}},\quad u_2(t)=\frac{tu^{-}+(1-t)u^{+}}{\|tu^{-}+(1-t)u^{+}\|_{L^2}},\]
\[u_3(t)=\frac{-tu^{-}+(1-t)u}{\|-tu^{-}+(1-t)u\|_{L^2}}.\]
Then by using above calculation one can easily get that for all
$t\in[0,1]$,
\begin{align*}
\tilde{J_{p}}(u_1(t))&=
\frac{\int_{Q}[(u^+(x)-u^+(y))^2+(1-t)^2(u^-(x)-u^-(y))^2]
K(x-y)dxdy}{\|u^{+}-(1-t)u^{-}\|^{2}_{L^2}}\\
&\quad\quad\quad +\frac{4(1-t)\int_{Q}u^-(x)u^-(y)K(x-y)
dxdy-p\int_{\Om}(u^{+})^2 dx}{\|u^{+}-(1-t)u^{-}\|^{2}_{L^2}}\\
&= \mu -\frac{2 t^2 \int_{Q} u^+(x)u^-(y)K(x-y)
dxdy}{\|u^{+}-(1-t)u^{-}\|^{2}_{L^2}}.
\end{align*}
\begin{align*}
\ds \tilde{J_{p}}(u_2(t))&=\frac{\int_{Q}[(1-t)^2(u^+(x)-u^+(y))^2+
t^2(u^-(x)-u^-(y))^2]
K(x-y)dxdy}{\|(1-t)u^{+}+tu^{-}\|^{2}_{L^2}}\\
&-\frac{4t(1-t)\int_{Q}u^+(x)u^-(y)K(x-y)
dxdy+p(1-t)^2\int_{\Om}(u^{+})^2 +pt^2\int_{\Om} (u^-)^2}{\|(1-t)u^{+}+tu^{-}\|^{2}_{L^2}}\\
&= \mu -\frac{2 \int_{Q} u^+(x)u^-(y)K(x-y) dxdy}{\|(1-t)u^+ +
tu^{-}\|^{2}_{L^2}}- \frac{p t^2 \int_{\Om}(u^{-})^2 dx}
{\|(1-t)u^{+}+ t u^{-}\|^{2}_{L^2}}.
\end{align*}

\begin{align*}
\tilde{J_{p}}(u_3(t))&=
\frac{\int_{Q}[(1-t)^2(u^+(x)-u^+(y))^2+(u^-(x)-u^-(y))^2]
K(x-y)dxdy}{\|(1-t)u^{+}-u^{-}\|^{2}_{L^2}}\\
&\quad\quad\quad +\frac{4(1-t)\int_{Q}u^+(x)u^-(y)K(x-y)
dxdy-p(1-t)^2\int_{\Om}(u^{+})^2 dx}{\|(1-t)u^{+}-u^{-}\|^{2}_{L^2}}\\
&= \mu -\frac{2 t^2 \int_{Q} u^+(x)u^{-}(y)K(x-y) dxdy}{\|(1-t)u^{+}
-u^-\|^{2}_{L^2}}.
\end{align*}
 Let $\mc O = \{v\in\mc P: \tilde{J_p}(v)<\mu-p\}$. Then clearly $\phi_1\in
\mc O$, while $-\phi_1\in \mc O$ if $\mu- p>\la_1$. Moreover
$\phi_1$ and $-\phi_1$ are the only possible critical points of
$\tilde{J_p}$ in $\mc O$ because of the choice of $\mu$.

\noi We note that
$\tilde{J_p}\left(\frac{u^-}{\|u^-\|_{L^2}}\right)\leq \mu-p$,
$\frac{u^-}{\|u^-\|_{L^2}} $ does not change sign and vanishes on a
set of positive measure, it is not a critical point of
$\tilde{J_p}$. Therefore there exists a $C^1$ path $\eta:[-\e,\e]\ra
\mc P$ with $\eta(0)= \frac{u^-}{\|u^-\|_{L^2}}$ and
$\frac{d}{dt}\tilde{J_p}(\eta(t))|_{t=0}\ne 0$. Using this path we
can move from $\frac{u^-}{\|u^-\|_{L^2}}$ to a point $v$ with
$\tilde{J_p}(v)<\mu-p$. Taking a connected component of $\mc O$
containing $v$ and applying Lemma \ref{le01} we have that either
$\phi_1$ or $-\phi_1$ is in this component. Let us assume that it is
$\phi_1$. So we continue by a path $u_{4}(t)$ from
$\left(\frac{u^-}{\|u^-\|_{L^2}}\right)$ to $\phi_1$ which is at
level less than $\mu$. Then the path $-u_{4}(t)$ connects
$\left(-\frac{u^-}{\|u^-\|_{L^2}}\right)$ to $-\phi_1$. We observe
that
\[|\tilde{J_p}(u)- \tilde{J_p}(-u)|\leq p.\]
Then it follows that
\[\tilde{J_p}(-u_4(t))\leq \tilde{J_p}(u_4(t))+p\leq \mu-p+p= \mu \;\forall\; t.\]
Connecting $u_1(t)$, $u_2(t)$ and $u_4(t)$, we get a path from $u$
to $\phi_1$ and joining $u_3(t)$ and $-u_4(t)$ we get a path from
$u$ to $-\phi_1$. These yields a path $\ga(t)$ on $\mc P$ joining
from $-\phi_1$ to $\phi_1$ such that $\tilde{J_p}(\ga(t))\leq \mu$
for all $t$, which concludes the proof.\QED
\end{proof}
\begin{Corollary}
The second eigenvalue $\la_2$ of \eqref{eq02} has the variational
characterization given as
\[\la_2=\inf_{\ga\in\Ga}\max_{u\in\ga[-1,1]}\int_{Q}(u(x)-u(y))^2K(x-y) dxdy,\]
where $\Ga$ is same as in Proposition \ref{prop2.8}.
\end{Corollary}

\begin{proof}
Take $s=0$ in Theorem \ref{3.4}. Then we have $c(0)=\la_2$ and
\eqref{eq18} concludes the proof.\QED
\end{proof}
\section{Properties of the curve}
In this section we prove that the curve $\mc C$ is Lipschitz
continuous, has a certain asymptotic behavior and strictly
decreasing.

\begin{Proposition}
The curve $p\ra (p+c(p), c(p))$, $p\in \mb R^+$ is Lipschitz
continuous.
\end{Proposition}

\begin{proof} Proof follows as in Proposition 4.1 of
\cite{cfg}. For completeness we give details. 
Let $p_1<p_2$ then $\tilde{J}_{p_1}(u)>\tilde{J}_{p_2}(u)$ for all
$u\in \mc P$. So we have $c(p_1)\geq c(p_2)$. Now for every $\e>0$
there exists $\ga\in \Gamma$ such that
\[\max_{u\in \ga[-1,1]}\tilde{J}_{p_2}(u)\leq c(p_2)+\e,\]
and so
\[0\leq c(p_1)- c(p_2)\leq \max_{u\in \ga[-1,1]}\tilde{J}_{p_1}(u)- \max_{u\in \ga[-1,1]}\tilde{J}_{p_2}(u)+\e.\]
Let $u_0\in\ga[-1,1]$ such that
\[\max_{u\in \ga[-1,1]}\tilde{J}_{p_1}(u)=\tilde{J}_{p_1}(u_0)\]
then
\[0\leq c(p_1)- c(p_2)\leq \tilde{J}_{p_1}(u_0)-\tilde{J}_{p_2}(u_0)+\e \leq p_2-p_1+\e,\]
as $\e>0$ is arbitrary so the curve $\mc C$ is Lipschitz continuous
with constant $\leq 1$.
\end{proof}
\begin{Lemma}\label{111}
Let $A$, $B$ be two bounded open sets in $\mb R^n$, with
$A\subsetneq B$ and $B$ is connected then $\la_1(A)>\la_{1}(B)$.
\end{Lemma}
\begin{proof} 
From the Theorem 1 and 2 of \cite{weak}, we see that $\phi_1$ is
continuous and is a solution of \eqref{eq02} in viscosity sense.
Then from Lemma 12 of \cite{peter}, $\phi_1>0$. Now the variational characterization we see that for $A\subset B,$  $\la_1(A)\geq \la_{1}(B)$. Since  $\phi_1 (B)>0$ in $B$, we get the strict inequality as claimed.
\end{proof}
\begin{Lemma}\label{le31}
Let $(\al,\ba)\in \mc C$, and let $\al(x)$, $\ba(x)\in
L^{\infty}(\Om)$ satisfying
\begin{align}\label{eq31}
\la_1\leq \al(x)\leq \al,\;\;\la_1\leq \ba(x)\leq \ba.
\end{align}
 Assume that
\begin{align}\label{eq32}
 \la_1<\al(x)\;\mbox{and}\; \la_1<\ba(x)\;\mbox{ on subsets of
positive measure}.
\end{align}
Then any non-trivial solution $u$ of
\begin{equation}\label{eq33}
 \quad -\mc L_{K}u = \al(x) u^{+} - \ba(x) u^{-} \; \text{in}\;
\Om, 
 \quad u = 0 \; \mbox{in}\; \mb R^n \setminus\Om.
\end{equation}
changes sign in $\Om$ and \[\al(x)=\al\; \mbox{a.e. on}\;\{x\in \Om
: u(x)>0\}, \quad \ba(x)=\ba\; \mbox{a.e. on}\;\{x\in \Om :
u(x)<0\}.\]
\end{Lemma}

\begin{proof}
Let $u$ be a nontrivial solution of \eqref{eq33}. Replacing $u$ by
$-u$ if necessary. we can assume that the point $(\al,\ba)$ in $\mc
C$ is such that $\al\geq\ba$. We first claim that $u$ changes sign
in $\Om$. Suppose by contradiction that this is not true, first
consider the case $u\geq 0$,(case $u\leq 0$ can be prove similarly).
Then $u$ solves
\[-\mc L_k u = \al(x) u \;\mbox{in}\;\Om\;\;\quad u=0\;\mbox{on}\;\mb R^n\setminus \Om.\]
This implies that the first eigenvalue of $-\mc L_K$ on $X_0$ with
respect to weight $\al(x)$ is equal to 1. i.e
\begin{align}\label{eq34}
\inf\left\{\frac{\int_{Q} (v(x)-v(y))^2 K(x-y) dxdy}{\int_{\Om}
\al(x)v^2 dx}: v\in X_0, v\ne 0\right\}=1.
\end{align}

\noi We deduce from \eqref{eq31}, \eqref{eq32} and \eqref{eq34} that
\[1=\frac{\int_{Q} (\phi_1(x)-\phi_1(y))^2 K(x-y) dxdy}{\la_1}>\frac{\int_{Q} (\phi_1(x)-\phi_1(y))^2 K(x-y) dxdy}{\int_{\Om} \al(x)\phi_{1}^2
dx}\geq 1, \] a contradiction and hence the claim.

\noi Again we assume by contradiction that either
\begin{align}\label{eq35}
|\{x\in X_0 : \al(x)<\al\; \mbox{and}\; u(x)>0\}|>0
\end{align}
or
\begin{align}\label{eq36}
|\{x\in X_0 : \ba(x)<\ba\; \mbox{and}\; u(x)<0\}|>0.
\end{align}
\noi Suppose that $\eqref{eq35}$ holds (a similar argument will hold
for \eqref{eq36}). Put $\al-\ba=p\geq 0$. Then $\ba= c(p)$, where
$c(p)$ is given by \eqref{eq18}. We show that there exists a path
$\ga\in \Ga$ such that
\begin{align}\label{eq37}
\max_{u\in \ga[-1,1]}\tilde{J_{p}}(u) < \ba,
\end{align}
which yields a contradiction with the definition of $c(p)$.

\noi In order to construct $\ga$ we show that there exists of a
function $v\in X_0$ such that it changes sign and satisfies
\begin{align}\label{eq38}
\frac{\int_{Q} (v^{+}(x)-v^{+}(y))^2 K(x-y) dxdy}{\int_{\Om}
(v^{+})^2 dx} <\al \quad\; \mbox{and}\quad\; \frac{\int_{Q}
(v^{-}(x)-v^{-}(y))^2 K(x-y) dxdy}{\int_{\Om} (v^{-})^2 dx} <\ba.
\end{align}
\noi For this let $\mc O_1$ be a component of $\{x\in\Om : u(x)>0\}$
satisfying
\[|x\in\mc O_1:\al(x)<\al|>0,\]
which is possible by \eqref{eq35}. Let $\mc O_2$ be a component of
$\{x\in\Om : u(x)<0\}$ satisfying
\[|x\in\mc O_1:\ba(x)<\ba|>0,\]
which is possible by \eqref{eq36}. Then we claim that
\begin{align}\label{eq39}
\la_1(\mc O_1)<\al \quad \mbox{and} \quad \la_1(\mc O_2)\leq \ba,
\end{align}
\noi where $\la_1(\mc O_i)$ denotes the first eigenvalue of $-\mc
L_k$ on $X_0|_{\mc O_i}=\{u\in X|_{\mc O_i}: u=0\; \mbox{on}\;\mb
R^n\setminus \mc O_i\}$. Clearly $u|_{\mc O_i}\in X_0|_{\mc O_i}$
then we have
\[\frac{\int_{ Q|_{\mc O_1}} (u(x)-u(y))^2 K(x-y) dxdy}{\int_{\mc O_1} u^2 dx}<
\al\frac{\int_{ Q|_{\mc O_1}} (u(x)-u(y))^2 K(x-y) dxdy}{\int_{\mc
O_1} \al(x)u^2 dx} =\al\] which implies that $\la_1(\mc O_1)<\al$.
The other inequality in \eqref{eq39} is proved similarly. Now with
some modification on the sets $\mc O_1$ and $\mc O_2$, we construct
the sets $\tilde{\mc O}_1$ and $\tilde{\mc O}_2$ such that
$\tilde{\mc O}_1\cap \tilde{\mc O}_2=\emptyset$ and
$\la_1(\tilde{\mc O}_1)<\al$ and $\la_1(\tilde{\mc O}_2)<\ba$. For
$\nu\geq 0$, ${\mc O}_1(\nu)=\{x\in {\mc O}_1 : dist(x,{\mc
O}_{1}^{c})>\nu\}.$ Clearly $\la_1(\mc O_1(\nu))\geq \la_{1}(\mc
O_1))$ and moreover $\la_{1}(\mc O_1(\nu))\ra \la_{1}(\mc O_1))$ as
$\nu\ra0$. Then there exists $\nu_0>0$ such that
\begin{align}\label{eq40}
\la_1(\mc O_1(\nu))<\al\;\mbox{ for all }\;0\leq \nu\leq \nu_0.
\end{align}
Let $x_0\in \partial\mc O_2\cap \Om$ (is not empty as $\mc O_1\cap
\mc O_2=\emptyset)$ and choose $0<\nu<\min\{\nu_0,
dist(x_0,\Om^c)\}$ and $\tilde{\mc O}_1={\mc O}_1(\nu)$ and
$\tilde{\mc O}_2= {\mc O}_2\cap B(x_0, \frac{\nu}{2})$. Then
$\tilde{\mc O}_1\cap \tilde{\mc O}_2=\emptyset$ and by \eqref{eq40},
$\la_1(\tilde{\mc O}_1)<\al$. Since $\tilde{\mc O}_2$ is connected
then by \eqref{eq39} and Lemma \ref{111}, we get $\la_1(\tilde{\mc
O}_2)<\ba$. Now we define $v=v_1-v_2$, where $v_i$ are the the
extension by zero outside $\tilde{\mc O}_i$ of the eigenfunctions
associated to $\la_i(\tilde{\mc O}_i)$. Then $v$ satisfies
\eqref{eq38}. Thus there exist $v\in X_0$ which changes sign and
satisfies condition \eqref{eq38}, and moreover we have
\begin{align*}
\tilde{J_p}\left(\frac{v}{\|v\|_{L^2}}\right)&=
\frac{\int_{Q}(v^{+}(x)-v^{+}(y))^2 K(x-y) dxdy}{\|v\|^{2}_{L^2}}
+\frac{\int_{Q}(v^{-}(x)-v^{-}(y))^2 K(x-y) dxdy}{\|v\|^{2}_{L^2}}\\
&\quad -2\frac{\int_{Q}v^{+}(x)v^{-}(y) K(x-y) dxdy}{\|v\|^{2}_{L^2}}-p\frac{\int_{\Om}(v^{+})^2 dx}{\|v\|^{2}_{L^2}}\\
&< (\al-p)\frac{\int_{\Om}(v^{+})^2 dx}{\|v\|^{2}_{L^2}}+\ba
\frac{\int_{\Om}(v^{-})^2 dx}{\|v\|^{2}_{L^2}}-
2\frac{\int_{Q}v^{+}(x)v^{-}(y) K(x-y) dxdy}{\|v\|^{2}_{L^2}}<\ba.
\end{align*}
\[\tilde{J_p}\left(\frac{v^+}{\|v^+\|_{L^2}}\right)< \al-p =\ba,\quad\tilde{J_p}\left(\frac{v^-}{\|v^-\|_{L^2}}\right)< \ba-p. \]

\noi Using Lemma \ref{le01}, we have that there exists a critical
point in the connected component of the set $\mc O=\{u\in\mc
P:\tilde{J_{p}}(u)<\ba-p\}$. As the point $(\al,\ba)\in \mc C$, the
only possible critical point is $\phi_1$, then we can construct a
path from $-\phi_1$ to $\phi_1$ exactly in the same manner as in
Theorem \ref{3.4} only by taking $v$ in place of $u$. Thus we have
construct a path satisfying \eqref{eq37}, and hence the result
follows.\QED
\end{proof}

\begin{Corollary}\label{le32}
Let $(\al,\ba)\in \mc C$ and let $\al(x),\ba(x)\in L^{\infty}(\Om)$
satisfying $\la_1\leq \al(x)\leq \al$ a.e, $\la_1\leq \ba(x)\leq
\ba$ a.e. Assume that $\la_1<\al(x)$ and $\la_1<\ba(x)$ on subsets
of positive measure. If either $\al(x)<\al$ a.e in $\Om$ or
$\ba(x)<\ba$ a.e. in $\Om$. Then \eqref{eq33} has only the trivial
solution.
\end{Corollary}
\begin{Lemma}
The curve $p\ra (p+c(p), c(p))$ is strictly decreasing, $($in the
sense that $p_1<p_2$ implies $p_1+c(p_1)<p_2+c(p_2)$ and
$c(p_1)>c(p_2))$.
\end{Lemma}

\begin{proof}
 Let $p_1<p_2$ and suppose by contradiction that either
$(i)$ $p_1+c(p_1)\geq p_2+ c(p_2)$ or $(ii)$ $c(p_1)\leq c(p_2)$. In
case $(i)$ we deduce from $p_1+c(p_1)\geq p_2+ c(p_2)> p_1+c(p_2)$
that $c(p_1)\geq c(p_2)$. If we take $(\al, \ba)= (p_1+c(p_1),
c(p_1))$ and $(\al(x),\ba(x))= (p_2+c(p_2), c(p_2))$, then by
Corollary \ref{le32}, only solution of \eqref{eq33} with
$(\al(x),\ba(x))$ is the trivial solution which contradicts the fact
that $(p_2+ c(p_2),c(p_2))\in\sum_{K}$. If $(ii)$ holds then
$p_1+c(p_1)\leq p_1+ c(p_2)<p_2+c(p_2),$ if we take $(\al, \ba)=
(p_2+c(p_2), c(p_2))$ and $(\al(x),\ba(x))= (p_1+c(p_1), c(p_1))$,
then only solution of \eqref{eq33} with $(\al(x),\ba(x))$ is the
trivial one which contradicts the fact that $(p_1+
c(p_1),c(p_1))\in\sum_{K}$ and hence the result follows.\QED
\end{proof}

As $c(p)$ is decreasing and positive so limit of $c(p)$ exists as
$p\ra\infty.$ In the next Theorem we find the asymptotic behavior of
the first nontrivial curve.
\begin{Theorem}
If $n\geq 2s$ then the limit of $c(p)$ as $p\ra \infty$ is $\la_1$.
\end{Theorem}
\begin{proof}
For $n\geq2s$, we can choose a function $\phi\in X_0$ such that
there does not exist $r\in\mb R$ such that $\phi(x)\leq r \phi_1(x)$
a.e. in $\Om$. For this it suffices to take $\phi\in X_0$ such that
it is unbounded from above in a neighborhood of some point $x\in
X_0$. Then by contradiction argument, one can similarly show
$c(p)\ra \la_1$ as $p\ra \infty$ as in Proposition 4.4 of
\cite{cfg}. \QED

\end{proof}

\section{Non Resonance between $(\la_1,\la_1)$ and $\mc C$}
In this section we study the following problem
\begin{align}\label{041}
\left\{
\begin{array}{lr}
-\mc L_{K} u= f(x,u)\;\mbox{in}\;\Om\\
\quad u=0\; \mbox{on}\; \mb R^n\setminus\Om,
\end{array}
\right.
\end{align}
 where $f(x,u)/u$ lies asymptotically between
$(\la_1,\la_1)$ and $(\al,\ba)\in \mc C$. Let $f:\Om\times\mb R\ra
\mb R$ be a function satisfying $L^{\infty}(\Om)$ Caratheodory
conditions.
 Given a point $(\al,\ba)\in \mc C$, we assume following:
\begin{align}\label{41}
 \ga_{\pm}(x)\leq \liminf_{s\ra \pm\infty}\frac{f(x,s)}{s}\leq \limsup_{s\ra\pm \infty} \frac{f(x,s)}{s}\leq \Ga_{\pm}(x)
 \end{align}
hold uniformly with respect to $x$, where $\ga_{\pm}(x)$ and
$\Ga_{\pm}(x)$ are $L^{\infty}$ functions which satisfy
\begin{align}\label{42}
\left\{
\begin{array}{lr}
\quad \la_1\leq \ga_{+}(x)\leq \Ga_{+}(x)\leq\al\;\mbox{a.e. in}\;\Om\\
\quad \la_1\leq \ga_{-}(x)\leq \Ga_{-}(x)\leq\ba\;\mbox{a.e.
in}\;\Om.
\end{array}
\right.
\end{align}
Write $F(x,s)=\int_{0}^{s}f(x,t) dt$, we also assume the following
inequalities:
\begin{align}\label{43}
 \de_{\pm}(x)\leq \liminf_{s\ra \pm\infty}\frac{2F(x,s)}{s^2}\leq \limsup_{s\ra\pm \infty} \frac{2F(x,s)}{s^2}\leq \De_{\pm}(x)
 \end{align}
hold uniformly with respect to $x$, where $\de_{\pm}(x)$ and
$\De_{\pm}(x)$ are $L^{\infty}$ functions which satisfy
\begin{equation}\label{44}
\left\{
\begin{array}{lr}
\quad \la_1\leq \de_{+}(x)\leq \De_{+}(x)\leq\al\;\mbox{a.e. in}\;\Om\\
\quad \la_1\leq \de_{-}(x)\leq \De_{-}(x)\leq\ba\;\mbox{a.e. in}\;\Om\\
\quad \de_{+}(x)>\la_1 \;\mbox{and}\;\de_{-}(x)>\la_1 \;\mbox{on subsets of positive measure,}\\
\quad \mbox{either}\; \De_{+}(x)<\al \;\mbox{a.e.
in}\;\Om\;\mbox{or}\;\De_{-}(x)<\ba \;\mbox{a.e. in}\;\Om.
\end{array}
\right.
\end{equation}
\begin{Theorem}\label{th51}
Let \eqref{41}, \eqref{42}, \eqref{43}, \eqref{44} hold and
$(\al,\ba)\in\mc C$. Then problem \eqref{041} admits at least one
solution $u$ in $X_0$.
\end{Theorem}
Define the energy functional $\Psi:X_0\ra \mb R$ as
\[\Psi(u)=\frac{1}{2}\int_{Q} (u(x)-u(y))^2K(x-y)dxdy -\int_{\Om}F(x,u) dx\]
Then $\Psi$ is a $C^{1}$ functional on $X_0$ and $\fa v\in X_0$
\[\ld\Psi^{\prime}(u),v\rd=\int_{Q} (u(x)-u(y))(v(x)-v(y))K(x-y)dxdy -\int_{\Om}f(x,u)v dx\]
and critical points of $\Psi$ are exactly the weak solutions of
\eqref{041}.
\begin{Lemma}\label{le51}
 $\Psi$ satisfies the $(P.S)$ condition on $X_0$.
 \end{Lemma}
\begin{proof}
Let $u_k$ be a (P.S) sequence in $X_0$, i.e
\begin{align}
|\Psi(u_k)|&\leq c,\notag\\
 |\ld\Psi^{\prime}(u_k),v\rd|&\leq \e_k\|v\|_{X_0}, \;\fa\; v\in
X_0,\label{eq46}
\end{align}
where $c$ is a constant and $\e_k\ra 0$ as $k\ra\infty$. It suffices
to show that $u_k$ is a bounded sequence in $X_0$. Assume by
contradiction that $u_k$ is not a bounded sequence. Then define
$v_k= \frac{u_k}{\|u_k\|_{X_0}}$. Then $v_k$ is a bounded sequence.
Therefore there exists a subsequence $v_k$ of $v_k$ and $v_0\in X_0$
such that $v_k\rightharpoonup v_0 $ weakly in $X_0$, $v_k\ra v_0$
strongly in $L^{2}(\Om)$ and $v_{k}(x)\ra v_{0}(x)$ a.e. in $\mb
R^n$. Also by using \eqref{41} and \eqref{42}, we have $f(x,u_k)/
\|u_k\|_{X_0}\rightharpoonup f_{0}(x)$ weakly in $L^2(\Om)$. Take
$v=v_k-v_0$ in \eqref{eq46} and divide by $\|u_k\|_{X_0}$ we get
$v_k\ra v_0$. In particular $\|v_0\|_{X_0}=1$. One can easily seen
from \eqref{eq46} that
\[\int_{Q} (v_0(x)-v_0(y))(v(x)-v(y))K(x-y)dxdy -\int_{\Om}f_0(x)v dx=0\; \fa\; v\in X_0.\]

\noi Now by standard argument based on assumption \eqref{41},
$f_0(x)=\al(x)v_{0}^{+}-\ba(x)v_{0}^{-}$ for some $L^{\infty}$
functions $\al(x)$, $\ba(x)$ satisfying \eqref{eq31}. In the
expression of $f_{0}(x)$, the value of $\al(x)$ (resp.$\ba(x)$) on
$\{x: v_0(x)\leq 0\}$ $($resp. $\{x: v_0(x)\geq 0\})$ are
irrelevant, and consequently we can assume that
\begin{align}\label{eq47}
\al(x)>\la_1\;\mbox{on}\;\{x:
v_0(x)\leq0\}\;\mbox{and}\;\ba(x)>\la_1\;\mbox{on}\;\{x:
v_0(x)\geq0\}.
\end{align}
So $v_0$ is a nontrivial solution of equation \eqref{eq33}. It then
follows from Lemma \ref{le31} that either $(i)$ $\al(x)= \la_1$ a.e
in $\Om$ or $(ii)$ $\ba(x)= \la_1$ a.e in $\Om$, or
 $(iii)$ $v_0$ is an eigenfunction associated to the point $(\al,\ba)\in \mc
 C$. We show that in each cases we get a contradiction.
If $(i)$ holds then by \eqref{eq47}, $v_0>0$ a.e. in $\Om$ and
\eqref{eq33} gives $\int_{Q}(v_0(x)-v_0(y))^2 K(x-y)dxdy =
\la_1\int_{\Om} v_{0}^{2}$, which implies that $v_0$ is a multiple
of $\phi_1$. Dividing \eqref{eq46} by $\|u_k\|^{2}_{X_0}$ and taking
limit we get,
\[\la_1\int_{\Om} v_{0}^{2}=\int_{Q}(v_0(x)-v_0(y))^2 K(x-y)dxdy= \lim_{k\ra\infty}\int_{\Om}
 \frac{2F(x,u_k)}{\|u_k\|^{2}_{X_0}}\geq \int_{\Om}\de_{+}(x)v_{0}^{2} dx.\]
 This contradicts assumption \eqref{44}. The case $(ii)$ is treated similarly.
 Now if $(iii)$ holds, we deduce from \eqref{43} that
 \begin{align*}
\int_{\Om}\al (v_{0}^{+})^2+\ba (v_{0}^{-})^2
&=\int_{Q}(v_0(x)-v_0(y))^2 K(x-y)dxdy = \lim_{k\ra\infty}\int_{\Om}
\frac{2F(x,u_k)}{\|u_k\|^{2}_{X_0}}\\
&\leq\int_{\Om}\De_{+}(x)(v_{0}^{+})^{2}+\De_{-}(x)(v_{0}^{-})^{2}.
\end{align*}
 This contradicts assumption \eqref{44}, since $v_0$ changes sign.
 Hence $u_{k}$ is bounded sequence in $X_0$.\QED
\end{proof}


Now we study the geometry of $\Psi$.

\begin{Lemma}\label{le52}
There exists $R>0$ such that
\begin{align}\label{eq48}
\max\{\Psi(R\phi_1),\Psi(-R\phi_1)\} < \max_{u\in \ga[-1,1]}\Psi(u)
\end{align}
for any $\ga\in \Ga_1:=\{\ga\in C([-1,1],X_0) : \ga(\pm 1)= \pm
R\phi_1\}$.
\end{Lemma}
\begin{proof}
From \eqref{43}, we have for any $\e>0$ there exists $a_{\e}(x)\in
L^{2}(\Om)$ such that for a.e $x$,
\begin{equation}\label{eq49}
\ds \left\{
\begin{array}{lr}
\quad(\de_{+}(x)-\e)\frac{s^2}{2}-a_{\e}(x)\leq F(x,s)\leq (\De_{+}(x)+\e)\frac{s^2}{2}+a_{\e}(x)\;\fa\;s>0\\
\quad(\de_{-}(x)-\e)\frac{s^2}{2}-a_{\e}(x)\leq
F(x,s)\leq(\De_{-}(x)+\e)\frac{s^2}{2}+a_{\e}(x)\;\fa\;s<0 .
\end{array}
\right.
\end{equation}
Now consider the following functional associated to the functions
$\De_{\pm}(x)$ as
\[\Phi(u)=\int_{Q} (u(x)-u(y))^2K(x-y)dxdy -\int_{\Om}\De_{+}(x) (u^{+})^2 dx- \int_{\Om}\De_{-}(x) (u^{-})^2 dx.\]
Then we claim that
\begin{align}\label{eq410}
d= \inf_{\ga \in\Ga}\max_{u\in \ga[-1,1]} \Psi(u)>0
\end{align}
where $\Ga$ is the set of all continuous paths from $-\phi_1$ to
$\phi_1$ in $\mc P$. Write $p=\al-\ba\geq 0$. we can choose
$(\al,\ba)\in \mc C$ such that $\al\geq\ba$ (as replacing $u$ by
$-u$ if necessary), we have for any $\ga\in \Ga$
\begin{align*}
&\max_{u\in \ga[-1,1]} \tilde{J_{p}}(u)\geq c(p)=\ba.\\
\mbox{i.e.} \max_{u\in \ga[-1,1]} \left(\int_{Q}
(u(x)-u(y))^2\right.&\left.K(x-y)dxdy-\int_{\Om}\al(u^{+})^2 dx-
\int_{\Om}\ba (u^{-})^2 dx\right)\geq 0
\end{align*}
 which implies
\[\max_{u\in \ga[-1,1]}\Phi(u)\geq 0,\]
by \eqref{44}. So $d\geq 0$. On the other hand, since
$\de_{\pm}(x)\leq \De_{\pm}(x)$,
\[\Phi(\pm\phi_1)\leq \int_{\Om} (\la_1 -\de_{\pm}(x))\phi_{1}^{2} dx <0\]
by \eqref{44}. Thus we have a mountain pass geometry for the
restriction $\tilde{\Phi}$ of $\Phi$ to $\mc P$,
\[\max\{\tilde{\Phi}(\phi_1),\tilde{\Phi}(-\phi_1)\}<0\leq \max_{u\in \ga[-1,1]}\tilde{\Phi}(u)\]
for any path $\ga\in \Ga$ and moreover one can verify exactly as in
Lemma \ref{le21} that ${\Phi}$ satisfies the $(P.S.)$ condition on
$X_0$. Then $d$ is a critical value of $\tilde{\Phi}$ i.e there
exists $u\in \mc P$ and $\mu\in \mb R$ such that
\begin{equation*}
\left\{
\begin{array}{lr}
\quad \Phi(u)=d\\
\quad \ld \Phi^{\prime}(u),v \rd =\mu \ld I^{\prime}(u),v
\rd\;\fa\;v\in X_0.
\end{array}
\right.
\end{equation*}
Assume by contradiction that $d=0$. Taking $v=u$ in above, we get
$\mu=0$ so $u$ is a nontrivial solution of
\[-\mc L_{K}u = \De_{+}(x) u^{+}- \De_{-}(x)u^{-} \;\mbox{in}\; \Om\quad u=0\; \mbox{in}\; \mb R^n\setminus \Om\]
Using \eqref{44}, we get a contradiction with Lemma \ref{le31} .
This completes the proof of claim.

\noi Next we show that \eqref{eq48} hold. From the left hand side of
inequality \eqref{eq49}, we have for $R>0$ and $\eta>0$,
\[\Psi(\pm R\phi_1)\leq \frac{R^2}{2}\int_{\Om}(\la_1-\de_{\pm}(x))\phi_{1}^{2}+\frac{\eta R^2}{2}+\|a_\eta\|_{L^1}, \]
Then ,$\Psi(\pm R\phi_1)\ra -\infty$ as $R\ra +\infty$, by
\eqref{44} and letting $\eta$ to be sufficiently small. Fix $\e$
with $0<\e<d$. We can choose $R=R(\e)$ so that
\begin{align}\label{eq412}
\Psi(\pm R\phi_1) < -\|a_{\e}\|_{L^1},
\end{align}
where $a_{\e}$ is associated to $\e$ using \eqref{eq49}. Consider a
path $\ga\in \Ga_{1}$. Then if $0\in \ga[-1,1]$, then by
\eqref{eq412},
\[\Psi(\pm R\phi_1) < -\|a_{\e}\|_{L^1}\leq 0=\Psi(0)\leq \max_{u\in\ga[-1,1]}\Psi(u),\]
so Lemma is proved in this case. If $0\not\in \ga[-1,1]$, then we
take the normalized path
$\tilde{\ga}(t)=\frac{\ga(t)}{\|\ga(t)\|_{L^2}}$ belongs to $\Ga$.
Since by \eqref{eq49},
\[\Psi(u)\geq \frac{\Phi(u)-\e\|u\|^{2}_{L^2}}{2}- \|a_{\e}\|_{L^1},\]
we obtain
\[\max_{\ga\in[-1,1]}\frac{2\Psi(u)+\e \|u\|^{2}_{L^2}+2\|a_{\e}\|_{L^1}}{\|u\|^{2}_{L^2}}\geq \max_{\tilde{\ga}\in[-1,1]} \Phi(v)\geq d,\]
and consequently, by choice of $\e$,
\[\max_{\ga\in[-1,1]}\frac{2\Psi(u)+2\|a_{\e}\|_{L^1}}{\|u\|^{2}_{L^2}}\geq d-\e>0,\]
This implies that
\[\max_{u\in\ga[-1,1]} \Psi(u) > -\|a_{\e}\|_{L^1} > \Psi(\pm R\phi_1),\]
by \eqref{eq412} and hence the Lemma.
\end{proof}

\noi {\bf Proof of Theorem \ref{th51}:} Lemmas \ref{le51} and
\ref{le52} completes the proof.


 \end{document}